\documentclass[oneside,12pt]{amsart}
\usepackage{amsmath, amsfonts,amsthm,times,graphics}

 \makeatletter
\renewcommand*\subjclass[2][2010]{%
  \def\@subjclass{#2}%
  \@ifundefined{subjclassname@#1}{%
    \ClassWarning{\@classname}{Unknown edition (#1) of Mathematics
      Subject Classification; using '2010'.}%
  }{%
    \@xp\let\@xp\subjclassname\csname subjclassname@#1\endcsname
  }%
}

 \makeatother
\newtheorem{theorem}{Theorem}[section]
\newtheorem{lemma}[theorem]{Lemma}
\newtheorem{corollary}[theorem]{Corollary}

\theoremstyle{definition}

 \makeatletter
\renewcommand*\subjclass[2][2010]{%
  \def\@subjclass{#2}%
  \@ifundefined{subjclassname@#1}{%
    \ClassWarning{\@classname}{Unknown edition (#1) of Mathematics
      Subject Classification; using '1991'.}%
  }{%
    \@xp\let\@xp\subjclassname\csname subjclassname@#1\endcsname
  }%
}

 \makeatother
\begin{document}

\title{An Extension of a Congruence by Tauraso}

\author{Romeo Me\v strovi\' c}
\address{Department of Mathematics,
Maritime Faculty, University of Montenegro, 
Dobrota 36, 85330 Kotor, Montenegro 
romeo@ac.me}

{\renewcommand{\thefootnote}{}
\setcounter{footnote}{0}}

\begin{abstract}
 For a positive integer $n$ let $H_n=\sum_{k=1}^{n}1/n$
be the $n$th  harmonic number.
In this note we prove that for any prime $p\ge 7$,
   $$
\sum_{k=1}^{p-1}\frac{H_k}{k^2}\equiv \sum_{k=1}^{p-1}\frac{H_k^2}{k}
\equiv\frac{3}{2p}\sum_{k=1}^{p-1}\frac{1}{k^2}\pmod{p^2}.
   $$
Notice that the first part of this congruence is recently proposed 
by R. Tauraso as a problem in Amer. Math.  Monthly. 
In our elementary proof of the second part of the above congruence
we use certain classical congruences modulo a prime and the square of a prime, 
some congruences involving harmonic numbers 
and a combinatorial identity due to V. Hern\'{a}ndez. 
   \end{abstract}
  \maketitle

\section{Introduction and Main Results}

Given  positive integers $n$ and $m$, the 
{\it harmonic numbers of order} $m$ are those rational numbers
$H_{n,m}$ defined as 
  $$
H_{n,m}=\sum_{k=1}^{n}\frac{1}{k^m}.
  $$
For simplicity, we will denote by
 $$
H_n:=H_{n,1}=\sum_{k=1}^n\frac{1}{k}
  $$   
the $n$th  {\it harmonic number} (we assume in addition that  
$H_{0,m}=H_0=0$).

Usually, here as always in the sequel, we consider the congruence relation 
modulo a prime $p$ extended to the ring of rational numbers
with denominators not divisible by $p$. 
For such fractions we put $m/n\equiv r/s \,(\bmod{\,p})$ 
if and only if $ms\equiv nr\,(\bmod{\,p})$, and the residue
class of $m/n$ is the residue class of $mn'$ where 
$n'$ is the inverse of $n$ modulo $p$.  

By  a problem proposed by R. Tauraso in \cite{t1}
and recently solved by D. B. Tyler \cite{ty}, for any prime $p\ge 7$,
  \begin{equation}\label{con1}
\sum_{k=1}^{p-1}\frac{H_k}{k^2}\equiv \sum_{k=1}^{p-1}\frac{H_k^2}{k}
\pmod{p^2}.
   \end{equation}
Further, R. Tauraso \cite[Theorem 2.3]{t3} proved
 \begin{equation}\label{con2}
\sum_{k=1}^{p-1}\frac{H_k}{k^2}\equiv 
-\frac{3}{p^2}\sum_{k=1}^{p-1}\frac{1}{k}\pmod{p^2}.
   \end{equation}
Tauraso's proof of (\ref{con2}) is based on
an identity due to V. Hern\'{a}ndez \cite{h} (see Lemma~\ref{l2.5}) 
and the congruence for triple harmonic sum modulo a prime 
due to Zhao \cite{z} (see (\ref{con49}) of Remarks in Section 2). 
In this note, 
we give an elementary proof of  (\ref{con2}) and
its extension as follows.

\begin{theorem}\label{t1.1} If  $p\ge 7$ is a prime, then
   \begin{equation}\label{con3}
\sum_{k=1}^{p-1}\frac{H_k}{k^2}\equiv \sum_{k=1}^{p-1}\frac{H_k^2}{k}
\equiv -\frac{3}{p^2}\sum_{k=1}^{p-1}\frac{1}{k}
\equiv\frac{3}{2p}\sum_{k=1}^{p-1}\frac{1}{k^2}\pmod{p^2}.
   \end{equation}
 \end{theorem}
Recall that Z. W. Sun in \cite{s} established  basic congruences modulo a prime
$p\ge 5$ for several sums of terms involving harmonic numbers.
In particular,  Sun established
$\sum_{k=1}^{p-1}H_k^r\,(\bmod{\, p^{4-r}})$ for $r=1,2,3$. 
Further generalizations of these congruences are recently obtained by 
Tauraso in \cite{t2}.

Recall that the 
{\it Bernoulli numbers} $B_k$ are defined by the generating function
   $$
\sum_{k=0}^{\infty}B_k\frac{x^k}{k!}=\frac{x}{e^x-1}.
  $$
It is easy to find the values $B_0=1$, $B_1=-\frac{1}{2}$, 
$B_2=\frac{1}{6}$, $B_4=-\frac{1}{30}$, and $B_n=0$ for odd $n\ge 3$. 
Furthermore, $(-1)^{n-1}B_{2n}>0$ for all $n\ge 1$. 
Applying a congruence given in \cite[Theorem 5.1(a)]{s1}
related to the sum $\sum_{k=1}^{p-1}1/k^2$ modulo $p^3$,  the congruence 
(\ref{con3}) in terms of Bernoulli numbers may be written
as follows.

  \begin{corollary}\label{c1.2} Let $p\ge 7$ be a prime. Then
    \begin{equation}\label{con4}
 \sum_{k=1}^{p-1}\frac{H_k^2}{k}
\equiv \sum_{k=1}^{p-1}\frac{H_k}{k^2}
\equiv 3\left(\frac{B_{2p-4}}{2p-4}-\frac{2B_{p-3}}{p-3}\right)
\pmod{p^2}. 
     \end{equation}
In particular, we have
   \begin{equation}\label{con5}
 \sum_{k=1}^{p-1}\frac{H_k^2}{k}
\equiv \sum_{k=1}^{p-1}\frac{H_k}{k^2}
\equiv B_{p-3}\pmod{p}. 
   \end{equation}
\end{corollary}

\noindent{\bf Remark.} Notice that 
the second congruence of (\ref{con5}) was obtained by Sun and Tauraso
\cite[the congruence (5.4)]{st} by using a standard technique expressing 
sum of powers  in terms of  Bernoulli numbers.  

Our proof of the second part of the congruence (\ref{con3}) given 
in the next section is entirely elementary.
It is based on  certain classical congruences modulo a prime 
and the square of a prime, two simple  congruences given by Z. W. Sun \cite{s} 
and two particular cases of a  combinatorial identity 
due to V. Hern\'{a}ndez \cite{h}.

\section{Proof of Theorem~\ref{t1.1}}

The following congruences by  Z. W. Sun given in his recent paper \cite{s}
are needid in the proof of Theorem~\ref{t1.1}.

\vspace{2mm}
\begin{lemma}\label{l2.1} 
  Let $p\ge 7$ be a prime. Then
    \begin{equation}\label{con8}
H_{p-k}\equiv H_{k-1}\pmod{p}
              \end{equation}
and
     \begin{equation}\label{con9}
(-1)^k{p-1\choose k}\equiv 1-pH_k+\frac{p^2}{2}(H_k^2-H_{k,2})\pmod{p^3}
              \end{equation}
for every $k=1,2,\ldots,p-1$.
\end{lemma} 
\begin{proof} The congruences (\ref{con8}) and (\ref{con9}) are in fact the 
congruences (2.1) and (2.2) in \cite[Lemma 2.1]{s}, respectively. 
  \end{proof}

The following well known result is a generalization of
{\it Wolstenholme's theorem} (see, e.g., \cite[Theorem 1]{a}
or \cite{gr}). 

\vspace{2mm}
\begin{lemma}\label{l2.2}
 {\rm  (\cite[Theorem 3]{b}).}  Let $m$ be a positive 
integer and let  $p$ be a prime such that $p\ge m+3$. Then
         \begin{equation}\label{con10}\begin{split}    
H_{p-1,m}\equiv\left\{
    \begin{array}{ll}
0 & \pmod{p}\quad {if\,\, m\,\, is\,\, even}\\
0 & \pmod{p^2}\quad {if\,\, m\,\, is\,\, odd}.
  \end{array}\right.
         \end{split}\end{equation}
In particular, for any prime $p\ge 5$, 
  \begin{equation}\label{con11}
H_{p-1}\equiv 0\,(\bmod{\,p^2})\quad (Wolstenholme's\,\, theorem),
   \end{equation} 
and for any prime $p\ge 7$,
$H_{p-1,3}\equiv 0\,(\bmod{\,p^2})$  
and $H_{p-1,2}\equiv H_{p-1,4}\equiv 0\,(\bmod{\,p})$.
  \end{lemma} 

\begin{lemma}\label{l2.3}
  Let $p\ge 7$ be a prime. Then
 \begin{equation}\label{con12}
\sum_{k=1}^{n}\frac{H_{k-1}}{k^3}\equiv \sum_{k=1}^{n}\frac{H_k}{k^3}
\equiv 0\pmod{p},
  \end{equation}
\begin{equation}\label{con13}
\sum_{k=1}^{p-1}\frac{H_k^3}{k}\equiv\frac{3}{2}
\sum_{k=1}^{p-1}\frac{H_k^2}{k^2}
\pmod{p},
  \end{equation}
  \begin{equation}\label{con14}
\sum_{k=1}^{n}\frac{H_{k-1,3}}{k}\equiv \sum_{k=1}^{n}\frac{H_{k,3}}{k}
\equiv 0\pmod{p},
  \end{equation}
and 
 \begin{equation}\label{con15}
\sum_{k=1}^{p-1}\frac{H_k}{k^2}\equiv \sum_{k=1}^{p-1}\frac{H_k^2}{k}
\pmod{p^2}.
\end{equation}
 \end{lemma} 

\begin{proof} 
By the congruence (\ref{con8}) from Lemma~\ref{l2.1},  
$H_k\equiv H_{p-k-1}\,(\bmod{\,p})$ for each $k=1,2,\ldots,p-1$
(notice that this is true for $k=p-1$ because $p\mid H_{p-1}$),
and therefore
     \begin{equation}\label{con16}
\sum_{k=1}^{p-1}\frac{H_{k-1}}{k^3}
=\sum_{k=1}^{p-1}\frac{H_{(p-k)-1}}{(p-k)^3}
\equiv \sum_{k=1}^{p-1}\frac{H_k}{(p-k)^3}\equiv 
- \sum_{k=1}^{p-1}\frac{H_k}{k^3}\pmod{p}.
   \end{equation}
Furthermore, using (\ref{con10}) with $m=4$ we get
   \begin{equation}\label{con17}
\sum_{k=1}^{p-1}\frac{H_{k-1}}{k^3}
=\sum_{k=1}^{p-1}\frac{\left(H_k-\frac{1}{k}\right)}{k^3}=
\sum_{k=1}^{p-1}\frac{H_k}{k^3}-\sum_{k=1}^{p-1}\frac{1}{k^4}
\equiv \sum_{k=1}^{p-1}\frac{H_{k}}{k^3}\pmod{p}.
    \end{equation}
From (\ref{con16}) and (\ref{con17}) it follows that
   \begin{equation*}
\sum_{k=1}^{p-1}\frac{H_{k-1}}{k^3}\equiv 
\sum_{k=1}^{p-1}\frac{H_k}{k^3}\equiv 0\pmod{p},
    \end{equation*}
which is actually (\ref{con12}).

Since $H_k=H_{k-1}+1/k$, for each $k=1,2,\ldots ,p-1$ we have
   \begin{equation*}\begin{split}
\frac{H_k^3}{k}-\frac{H_{k-1}^3}{k}
&=\frac{1}{k}(H_k-H_{k-1})
\left(H_k^2+H_k\left(H_k-\frac{1}{k}\right)+\left(H_k-\frac{1}{k}\right)^2
\right)\\
&=\frac{1}{k^2}\left(3H_k^2-3\frac{H_k}{k}+\frac{1}{k^2}\right)=
3\frac{H_k^2}{k^2}-3\frac{H_k}{k^3}+\frac{1}{k^4}.
  \end{split}\end{equation*}
The above identity, 
(\ref{con12}) and (\ref{con10}) of Lemma~\ref{l2.2} with $m=4$
yield
     \begin{equation}\label{con18}
\sum_{k=1}^{p-1}\frac{H_k^3}{k}-
\sum_{k=1}^{p-1}\frac{H_{k-1}^3}{k}=
3\sum_{k=1}^{p-1}\frac{H_k^2}{k^2}-3\sum_{k=1}^{p-1}\frac{H_k}{k^3}
+\sum_{k=1}^{p-1}\frac{1}{k^4}\equiv 3\sum_{k=1}^{p-1}\frac{H_k^2}{k^2}
\pmod{p}.
 \end{equation}
On the other hand, since by (\ref{con8}) from Lemma~\ref{l2.1},  
$H_k\equiv H_{p-k-1}\,(\bmod{\,p})$ for each $k=1,2,\ldots,p-1$,
then   
   \begin{equation}\label{con19}
\sum_{k=1}^{p-1}\frac{H_{k-1}^3}{k}
=\sum_{k=1}^{p-1}\frac{H_{(p-k)-1}^3}{p-k}
\equiv \sum_{k=1}^{p-1}\frac{H_k^3}{p-k} 
\equiv -\sum_{k=1}^{p-1}\frac{H_k^3}{k}\pmod{p}.
   \end{equation}
Taking (\ref{con19}) into (\ref{con18}) gives
  \begin{equation}\label{con20}
\sum_{k=1}^{p-1}\frac{H_k^3}{k}\equiv\frac{3}{2}
\sum_{k=1}^{p-1}\frac{H_k^2}{k^2}
\pmod{p},
    \end{equation}
which proves (\ref{con13}).

Proof of the congruence (\ref{con14}) is completely analogous to the previous
proof using the fact that by Lemma~\ref{l2.2}, 
$H_{p-1,3}\equiv 0\,(\bmod{\,p})$ 
for each $k=1,\ldots,p-1$, and therefore 
     $$
H_{k-1,3}=\sum_{i=1}^{k-1}\frac{1}{i^3}=
H_{p-1,3}-\sum_{j=1}^{p-k}\frac{1}{(p-j)^3}\equiv 
\sum_{j=1}^{p-k}\frac{1}{j^3}=H_{p-k,3}\pmod{p}.
   $$
Finally (cf. \cite{ty}), from the identity
     $$
H_{k-1}^3=\left(H_k-\frac{1}{k}\right)^3=H_k^3-3\frac{H_k^2}{k}+
3\frac{H_k}{k^2}-\frac{1}{k^3}
   $$
immediately follows that 
 \begin{equation}\label{con21}
\sum_{k=1}^{p-1}\frac{H_k^2}{k}-\sum_{k=1}^{p-1}\frac{H_k}{k^2}
=\frac{1}{3}\left(H_{p-1}^3-H_{p-1,3}\right).
   \end{equation}
Inserting in the right hand side of the  identity (\ref{con21})
the congruences $H_{p-1}\equiv H_{p-1,3}\equiv 0\,(\bmod{\,p^2})$
given in Lemma~\ref{l2.2}, we immediately obtain (\ref{con15}).
This  completes the proof.
\end{proof}

\begin{lemma}\label{l2.4}  
 Let $p\ge 7$ be a prime. Then
      \begin{equation}\label{con22} 
\sum_{k=1}^{p-1}\frac{H_k\cdot H_{k,2}}{k}\equiv 
\sum_{1\le i<j< k\le p-1}\frac{1}{ij^2k}+
\sum_{1\le i<j< k\le p-1}\frac{1}{i^2jk}\pmod{p}.
       \end{equation}
\end{lemma} 
\begin{proof} Since $H_k=H_{k-1}+1/k$ for every $k=1,2,\ldots,p-1$, we get
        \begin{equation}\label{con23}\begin{split}
\sum_{k=1}^{p-1}\frac{H_k\cdot H_{k,2}}{k}
&=\sum_{k=1}^{p-1}\frac{1}{k}\left(H_{k-1}+\frac{1}{k}\right)
\left(H_{k-1,2}+\frac{1}{k^2}\right)\\
&=\sum_{k=1}^{p-1}\frac{H_{k-1}\cdot H_{k-1,2}}{k}+
\sum_{k=1}^{p-1}\frac{H_{k-1,2}}{k^2}+
\sum_{k=1}^{p-1}\frac{H_{k-1}}{k^3}+\sum_{k=1}^{p-1}\frac{1}{k^4}.
       \end{split}\end{equation}
Using particular congruences given in Lemma~\ref{l2.2} with 
$m=2$ and $m=4$, we find that
  \begin{equation}\begin{split}\label{con24}
          \sum_{k=1}^{p-1}\frac{H_{k-1,2}}{k^2}
&=\sum_{k=1}^{p-1}\frac{1}{k^2}\sum_{i=1}^{k-1}\frac{1}{i^2}=
\sum_{1\le i<k\le p-1}\frac{1}{i^2k^2}\\
&=\frac{1}{2}\left(\sum_{k=1}^{p-1}\frac{1}{k^2}\right)^2-\frac{1}{2}
\sum_{k=1}^{p-1}\frac{1}{k^4}\equiv 0\pmod{p}.
    \end{split}\end{equation}
Substituting the congruences  (\ref{con24}), 
 (\ref{con12}) of Lemma~\ref{l2.3} and (\ref{con10}) with $m=4$ 
of Lemma~\ref{l2.2}  into (\ref{con23}),
we obtain 
   \begin{equation}\label{con25}
\sum_{k=1}^{p-1}\frac{H_k\cdot H_{k,2}}{k}
\equiv \sum_{k=1}^{p-1}\frac{H_{k-1}\cdot H_{k-1,2}}{k}\pmod{p}.
    \end{equation}
The right hand side of (\ref{con25}) can be expressed as  
      \begin{equation}\label{con26}\begin{split}
\sum_{k=1}^{p-1}\frac{H_{k-1}\cdot H_{k-1,2}}{k}
&=\sum_{k=1}^{p-1}\frac{1}{k}\left(1+\frac{1}{2}+\cdots +
\frac{1}{k-1}\right)\left(1+\frac{1}{2^2}+\cdots +
\frac{1}{(k-1)^2}\right)\\
&= \sum_{k=1}^{p-1}\frac{1}{k}\left(\sum_{1\le i<j\le k-1}\frac{1}{ij^2}+
\sum_{1\le i<j\le k-1}\frac{1}{i^2j}+\sum_{i=1}^{k-1}\frac{1}{i^3}\right)\\
&=\sum_{1\le i<j< k\le p-1}\frac{1}{ij^2k}+
\sum_{1\le i<j< k\le p-1}\frac{1}{i^2jk}+\sum_{k=1}^{p-1}\frac{H_{k-1,3}}{k}.
       \end{split}\end{equation}
Taking (\ref{con14}) of Lemma~\ref{l2.3} into
(\ref{con26}), and comparing this with (\ref{con25}), 
 we immediately obtain (\ref{con22}).
  \end{proof}

Further, for the proof of Theorem~\ref{t1.1} 
we will need two particular 
cases of the following  identity due to V. Hern\'{a}ndez \cite{h}. 

\begin{lemma}\label{l2.5} {\rm(\cite{h})}  Let $n$ and $m$ be positive integers.
Then 
    \begin{equation}\label{con27}
\sum_{k=1}^n{n\choose k}(-1)^{k-1}\sum_{1\le i_1\le i_2\le \cdots\le i_m=k}
\frac{1}{i_1i_2\cdots i_m}=\sum_{k=1}^n\frac{1}{k^m}.
   \end{equation}
  \end{lemma}

\begin{lemma}\label{l2.6}   Let $p\ge 7$ be a prime. Then
  \begin{equation}\label{con28}
\sum_{k=1}^{p-1}\frac{H_k\cdot H_{k,2}}{k}\equiv -\frac{3}{2}
\sum_{k=1}^{p-1}\frac{H_k^2}{k^2}\pmod{p}.
   \end{equation}
\end{lemma}
   \begin{proof}
The identity (\ref{con27}) of Lemma~\ref{l2.5}
with $m=3$ and $n=p-1$ becomes
  \begin{equation}\label{con29}
\sum_{k=1}^{p-1}\frac{(-1)^{k-1}}{k}{p-1\choose k}\sum_{1\le i\le j\le k}
\frac{1}{ij}=\sum_{k=1}^{p-1}\frac{1}{k^3}=H_{p-1,3}.
   \end{equation}
For any fixed $k\le p-1$, we have the identity
   \begin{equation}\label{con30}
\sum_{1\le i\le j\le k}\frac{1}{ij}=\frac{1}{2}
\left(\left(\sum_{i=1}^k\frac{1}{i}\right)^2+
\sum_{i=1}^k\frac{1}{i^2}\right)=\frac{1}{2}(H_k^2+H_{k,2}).
  \end{equation}
Next the congruence (\ref{con9}) from Lemma~\ref{l2.1}
reduced modulo $p^2$ gives
        \begin{equation}\label{con31}
(-1)^{k-1}{p-1\choose k}\equiv pH_k-1\pmod{p^2}
            \end{equation}
for every $k=1,2,\ldots,p-1$. Substituting (\ref{con30}), (\ref{con31})
and the congruence  $H_{p-1,3}\equiv 0\,(\bmod{\,p^2})$ 
of Lemma~\ref{l2.2} into (\ref{con29}), we immediately obtain 
         \begin{equation*}
\sum_{k=1}^{p-1}\frac{1}{k}(pH_k-1)(H_k^2+H_{k,2})\equiv 0\pmod{p^2},
     \end{equation*}
or equivalently,
  \begin{equation}\label{con32}
p\left(\sum_{k=1}^{p-1}\frac{H_k^3}{k}+
\sum_{k=1}^{p-1}\frac{H_k\cdot H_{k,2}}{k}\right)
-\left(\sum_{k=1}^{p-1}\frac{H_k^2}{k}
+\sum_{k=1}^{p-1}\frac{H_{k,2}}{k}\right)
\equiv 0\pmod{p^2}.
  \end{equation}
Further, (\ref{con15}) from Lemma~\ref{l2.3}   
and the congruences $H_{p-1}\equiv H_{p-1,3}\equiv 0\,(\bmod{\,p^2})$
from Lemma~\ref{l2.2} give
\begin{equation}\label{con33}\begin{split}
\sum_{k=1}^{p-1}\frac{H_k^2}{k}+\sum_{k=1}^{p-1}\frac{H_{k,2}}{k}
&\equiv \sum_{k=1}^{p-1}\frac{H_k}{k^2}+\sum_{k=1}^{p-1}\frac{H_{k,2}}{k}
\pmod{p^2}\\
&=\sum_{1\le i\le k\le p-1}\frac{1}{ik^2}
+\sum_{1\le i\le k\le p-1}\frac{1}{i^2k}\\
&=\left(\sum_{i=1}^{p-1}\frac{1}{i}\right)
\left(\sum_{k=1}^{p-1}\frac{1}{k^2}\right)+\sum_{k=1}^{p-1}\frac{1}{i^3}\\
&=H_{p-1}\cdot H_{p-1,2}+H_{p-1,3}\equiv 0\pmod{p^2}.
  \end{split}\end{equation}
Substituting (\ref{con33}) into (\ref{con32}), we find that
 \begin{equation}\label{con34}
\sum_{k=1}^{p-1}\frac{H_k^3}{k}+\sum_{k=1}^{p-1}\frac{H_k\cdot H_{k,2}}{k}
\equiv 0\pmod{p}.
   \end{equation}
Taking (\ref{con13}) of Lemma~\ref{l2.3} into 
(\ref{con34}) yields (\ref{con28}). This concludes the proof.
\end{proof}

\begin{lemma}\label{l2.7}   Let $p\ge 7$ be a prime. Then
      \begin{equation}\label{con35} 
   \sum_{1\le i<j< k\le p-1}\frac{1}{i^2jk}
\equiv\sum_{1\le i<j< k\le p-1}\frac{1}{ijk^2}
\equiv -\frac{1}{2}\sum_{1\le i<j< k\le p-1}\frac{1}{ij^2k}
\pmod{p}.
   \end{equation}
 \end{lemma} 
 \begin{proof}

For simplicity, we denote
   $$
A:=\sum_{1\le i<j< k\le p-1}\frac{1}{i^2jk},\quad
B:=\sum_{1\le i<j< k\le p-1}\frac{1}{ij^2k},\quad\mathrm{and}\quad
C:=\sum_{1\le i<j< k\le p-1}\frac{1}{ijk^2}. 
    $$
Obviously, holds the identity
 \begin{equation}\label{con36} 
\left(\sum_{1\le i<j<k\le p-1}\frac{1}{ijk}\right)
\left(\sum_{l=1}^{p-1}\frac{1}{l}\right)=A+B+C+
4\sum_{1\le i<j<k<l\le p-1}\frac{1}{ijkl}.
\end{equation}
The well known Newton's identities (see e.g., \cite{m}) imply 
     \begin{eqnarray*}
4\sum_{1\le i<j<k<l\le p-1}\frac{1}{ijkl}
&=&\sum_{1\le i<j<k\le p-1}\frac{1}{ijk}H_{p-1}-
\sum_{1\le i<j\le p-1}\frac{1}{ij}H_{p-1,2}\\
&&+\sum_{i=1}^{p-1}\frac{1}{i}H_{p-1,3}-H_{p-1,4},
  \end{eqnarray*}
whence since all the sums $H_{p-1}, H_{p-1,2}, H_{p-1,3}$ and $H_{p-1,4}$ 
are divisible by a prime $p\ge 7$, we obtain 
(cf. \cite[Theorem 1.5]{z} or \cite{zc})
  \begin{equation}\label{con37} 
\sum_{1\le i<j<k<l\le p-1}\frac{1}{ijkl}\equiv 0\pmod{p}.
   \end{equation}
Inserting (\ref{con37}) and 
$H_{p-1}=\sum_{i=1}^{p-1}1/i \equiv 0\,(\bmod{\,p})$
into (\ref{con36}), we get
   \begin{equation}\label{con38} 
A+B+C\equiv 0\pmod{p}.  
  \end{equation}
Further, by the substitution trick $i,j,k\to p-i,p-j,p-k$,
      \begin{equation}\label{con39}\begin{split}
A&=\sum_{1\le i<j< k\le p-1}\frac{1}{i^2jk}
=\sum_{1\le p-i<p-j< p-k\le p-1}\frac{1}{(p-i)^2(p-j)(p-k)}\\
&\equiv  \sum_{1\le k<j< i\le p-1}\frac{1}{kji^2}=C\pmod{p}.
  \end{split}\end{equation}
From (\ref{con39}) we see that $C\equiv A \,(\bmod{\,p})$, which substituting 
into (\ref{con38}) gives 
   \begin{equation}\label{con40} 
2A+B\equiv 0\pmod{p}.  
  \end{equation}
Finally, (\ref{con39}) and (\ref{con40}) yield 
$C\equiv A\equiv -B/2 \,(\bmod{\,p})$, 
as desired.
\end{proof}

\begin{lemma}\label{l2.8}   Let $p\ge 7$ be a prime. Then
  \begin{equation}\label{con41}
\sum_{k=1}^{p-1}\frac{H_k^2}{k^2}\equiv 
-\sum_{1\le i<j< k\le p-1}\frac{1}{ij^2k}\pmod{p}.
   \end{equation}
\end{lemma}
  \begin{proof}
We follow  proof of the congruence (3) in Theorem 1.1 of 
 \cite{me}. By  Lemma~\ref{l2.2},  
$H_{p-1}:=\sum_{j=1}^{p-1}1/j\equiv 0\,(\bmod{\,p})$,
or equivalently, for each $j=1,2,\ldots,p-2$ holds
 \begin{equation}\label{con42}
\frac{1}{j+1}+\frac{1}{j+2}+\cdots +\frac{1}{p-1}\equiv -\left(
1+\frac{1}{2}+\cdots +\frac{1}{j-1}+\frac{1}{j}\right)\pmod{p^2}.
 \end{equation}
Applying  the congruence (\ref{con42}), we find that 
    \begin{equation*}\begin{split}
&\sum_{1\le i<j< k\le p-1}\frac{1}{ij^2k}
=\sum_{j=2}^{p-1}\frac{1}{j^2}\sum_{i=1}^{j-1}\frac{1}{i}
\sum_{k=j+1}^{p-1}\frac{1}{k}\\
&=\sum_{j=2}^{p-1}\frac{1}{j^2}
\left(1+\frac{1}{2}+\cdots +\frac{1}{j-1}\right)
\left(\frac{1}{j+1}+\frac{1}{j+2}+\cdots +\frac{1}{p-1}\right)\\
&\equiv \sum_{j=2}^{p-1}\frac{1}{j^2}
\left(1+\frac{1}{2}+\cdots +\frac{1}{j-1}\right)
\left(-\left(1+\frac{1}{2}+\cdots +\frac{1}{j}\right)\right)\pmod{p}\\
&=\sum_{j=1}^{p-1}\frac{1}{j^2}\left(H_{j}-\frac{1}{j}\right)
\left(-H_j\right)=-\sum_{j=1}^{p-1}\frac{H_j^2}{j^2}+
\sum_{j=1}^{p-1}\frac{H_j}{j^3}\\
&=-\sum_{j=1}^{p-1}\frac{H_j^2}{j^2}+\sum_{j=1}^{p-1}\frac{H_{j-1}
+\frac{1}{j}}{j^3}=
-\sum_{j=1}^{p-1}\frac{H_j^2}{j^2}+\sum_{1\le i<j\le p-1}\frac{1}{ij^3}
+\sum_{j=1}^{p-1}\frac{1}{j^4}\\
&=-\sum_{j=1}^{p-1}\frac{H_j^2}{j^2}+\sum_{j=1}^{p-1}\frac{H_{j-1}}{j^3}
+H_{p-1,4}\hfill\qquad\qquad\qquad\pmod{p},
   \end{split}\end{equation*}
whence it follows that 
  \begin{equation}\label{con43}
 \sum_{j=1}^{p-1}\frac{H_j^2}{j^2}\equiv \sum_{j=1}^{p-1}\frac{H_{j-1}}{j^3}
+H_{p-1,4}-\sum_{1\le i<j< k\le p-1}\frac{1}{ij^2k}\pmod{p}.
  \end{equation}
Since by the first part of (\ref{con12}) from Lemma~\ref{l2.3} and
by (\ref{con10}) of Lemma~\ref{l2.2} with $m=4$,
  \begin{equation*}
\sum_{j=1}^{p-1}\frac{H_{j-1}}{j^3}\equiv
H_{p-1,4}\equiv 0\pmod{p},
  \end{equation*}
substituting this into (\ref{con43}), we obtain 
(\ref{con41}).
     \end{proof}
The first congruence of the following result was recently established by 
Z. W. Sun \cite[Theorem 1.1 (1.5)]{s}.

\begin{lemma}\label{l2.9}   Let $p\ge 7$ be a prime. Then
  \begin{equation}\label{con44}
\sum_{k=1}^{p-1}\frac{H_k^2}{k^2}\equiv 
0\pmod{p},
   \end{equation}
 \begin{equation}\label{con45}
\sum_{k=1}^{p-1}\frac{H_k\cdot H_{k,2}}{k}\equiv 0\pmod{p}
   \end{equation}
and
 \begin{equation}\label{con46}
\sum_{k=1}^{p-1}\frac{H_k^3}{k}\equiv 0\pmod{p}.
  \end{equation}

\end{lemma}
  \begin{proof}
Comparing the congruences (\ref{con22}) of Lemma~\ref{l2.4} and 
(\ref{con28}) of Lemma~\ref{l2.6}, we have 
        \begin{equation}\label{con47} 
\sum_{k=1}^{p-1}\frac{H_k^2}{k^2}\equiv 
-\frac{2}{3}\left(\sum_{1\le i<j< k\le p-1}\frac{1}{ij^2k}+
\sum_{1\le i<j< k\le p-1}\frac{1}{i^2jk}\right)\pmod{p}.
       \end{equation}
Since by (\ref{con35}) of Lemma~\ref{l2.7},
  \begin{equation*} 
   \sum_{1\le i<j< k\le p-1}\frac{1}{i^2jk}
\equiv -\frac{1}{2}\sum_{1\le i<j< k\le p-1}\frac{1}{ij^2k}
\pmod{p},
   \end{equation*}
then substituting this into (\ref{con47}), we obtain
 \begin{equation}\label{con48} 
\sum_{k=1}^{p-1}\frac{H_k^2}{k^2}\equiv 
-\frac{1}{3}\sum_{1\le i<j< k\le p-1}\frac{1}{ij^2k}\pmod{p}.
       \end{equation}
Finally,  as by (\ref{con41}) of Lemma~\ref{l2.8},
  \begin{equation*}
\sum_{k=1}^{p-1}\frac{H_k^2}{k^2}\equiv 
-\sum_{1\le i<j< k\le p-1}\frac{1}{ij^2k}\pmod{p},
   \end{equation*}
then comparing this with (\ref{con48}) implies 
  \begin{equation*}
\sum_{k=1}^{p-1}\frac{H_k^2}{k^2}\equiv 
0\pmod{p},
   \end{equation*}
which coincides with (\ref{con44}). 

Finally, (\ref{con44}) and (\ref{con28}) of Lemma~\ref{l2.6} 
yield (\ref{con45}), while (\ref{con44}) and (\ref{con13}) of Lemma~\ref{l2.3} 
yield (\ref{con46}).
  \end{proof}

 \noindent{\bf Remarks.} 
Applying a standard technique expressing 
sum of powers  in terms of  Bernoulli numbers, 
Z. W. Sun in \cite[Proof of (1.5) of Theorem]{s} showed that 
   \begin{equation*}
\sum_{k=1}^{p-1}\frac{H_k^2}{k^2}\equiv -\sum_{j=0}^{p-3}
B_{j}B_{p-3-j}\pmod{p}.
   \end{equation*}
The above congruence and  (\ref{con44}) yield 
the following  curious congruence for a prime $p\ge 7$ established by J. Zhao 
\cite[(3.19) of Corollary 3.6]{z}:
    $$
\sum_{j=0}^{p-3}B_{j}B_{p-3-j}\equiv 0\pmod{p}.
  $$
As noticed in \cite[Proof of Lemma 2.8]{s}, the above
congruence immediately follows from an identity of
Matiyasevich (cf. \cite[(1.3)]{ps}).

Furthermore, the  congruences (\ref{con44}), (\ref{con48}) 
and (\ref{con35}) from Lemma~\ref{l2.7} immediately give
  \begin{equation}\label{con49} 
   \sum_{1\le i<j< k\le p-1}\frac{1}{i^2jk}
\equiv\sum_{1\le i<j< k\le p-1}\frac{1}{ijk^2}
\equiv \sum_{1\le i<j< k\le p-1}\frac{1}{ij^2k}\equiv 0
\pmod{p}.
   \end{equation}
Notice that the congruences (\ref{con49})  were  proved by  J. Zhao 
\cite[Corollary 3.6 (3.20)]{z}
applying a technique expressing sum of powers in terms 
of Bernoulli numbers. 

\vspace{1mm}

The following result is contained in \cite[Lemma 2.4]{me2}.
\begin{lemma}\label{l2.10}   Let $p\ge 7$ be a prime. Then
  \begin{equation}\label{con50}
2\sum_{k=1}^{p-1}\frac{1}{k}\equiv 
-p\sum_{k=1}^{p-1}\frac{1}{k^2}\pmod{p^4}.
   \end{equation}
\end{lemma}
  \begin{proof}
Multiplying   the identity
  $$
1+\frac{p}{k}+\frac{p^2}{k^2}=\frac{p^3-k^3}{k^2(p-k)}
 $$
by $-p/k^2$ ($1\le k\le p-1$), we obtain
  $$   
-\frac{p}{k^2}\left(1+\frac{p}{k}+\frac{p^2}{k^2}\right)
=\frac{-p^4+pk^3}{k^4(p-k)}\equiv \frac{p}{k(p-k)}\pmod{p^4},
     $$   
which can be written as
 $$
\frac{1}{k}+\frac{1}{p-k}\equiv
-\left(\frac{p}{k^2}+\frac{p^2}{k^3}+\frac{p^3}{k^4}\right)
\pmod{p^4}.
 $$
After summation of the above congruence over 
$k=1,\ldots ,p-1$ we immediately obtain
  \begin{equation}\label{con51}
2H_{p-1}\equiv -pH_{p-1,2}-p^2H_{p-1,3}-p^3H_{p-1,4}\pmod{p^4}.
  \end{equation}
Taking the congruences $H_{p-1,3}\equiv 0\,(\bmod{\,p^2})$ 
and $H_{p-1,4}\equiv 0\,(\bmod{\,p})$ of Lemma~\ref{l2.2} into 
(\ref{con51}), it becomes (\ref{con50}).
  \end{proof}

\begin{proof}[Proof of Theorem~\ref{t1.1}]
The first congruence of (\ref{con3}) is in fact the congruence 
(\ref{con15}) of Lemma~\ref{l2.3}.

The identity (\ref{con27}) of Lemma~\ref{l2.5}
with $m=2$ and $n=p-1$ becomes
  \begin{equation}\label{con52}
\sum_{1\le i\le k\le p-1}(-1)^{k-1}{p-1\choose k}
\frac{1}{ik}=\sum_{k=1}^{p-1}\frac{1}{k^2}=H_{p-1,2}.
   \end{equation}
Substituting the congruence (\ref{con9}) of Lemma~\ref{l2.1}
into the left hand side of equality (\ref{con52}), we find that
   \begin{equation}\label{con53}\begin{split}
& \sum_{1\le i\le k\le p-1}\frac{(-1)^{k-1}}{ik}{p-1\choose k}
\equiv  -\sum_{1\le i\le k\le p-1}\frac{1}{ik}+
p\sum_{1\le i\le k\le p-1}\frac{H_k}{ik}\\
&-\frac{p^2}{2}\sum_{1\le i\le k\le p-1}\frac{H_k^2}{ik}
+\frac{p^2}{2}\sum_{1\le i\le k\le p-1}\frac{H_{k,2}}{ik}\pmod{p^3}.
  \end{split}\end{equation}
By Wolstenholme's theorem,
    \begin{equation}\label{con54}
\sum_{1\le i\le k\le p-1}\frac{1}{ik}=
\frac{1}{2}\left(H_{p-1}^2+H_{p-1,2}\right)\equiv 
\frac{1}{2}H_{p-1,2}\pmod{p^2}.
   \end{equation}
Next we have    
     \begin{equation}\label{con55}
\sum_{1\le i\le k\le p-1}\frac{H_k}{ik}=\sum_{k=1}^{p-1}\frac{H_k}{k}
\sum_{i=1}^{k}\frac{1}{i}=\sum_{k=1}^{p-1}\frac{H_k^2}{k}.
     \end{equation}
Further, using (\ref{con46}) of Lemma~\ref{l2.9},
\begin{equation}\label{con56}
\sum_{1\le i\le k\le p-1}\frac{H_k^2}{ik}=\sum_{k=1}^{p-1}\frac{H_k^2}{k}
\sum_{i=1}^{k}\frac{1}{i}=\sum_{k=1}^{p-1}\frac{H_k^3}{k}\equiv 0\pmod{p}.
     \end{equation}
Similarly, using (\ref{con45}) of Lemma~\ref{l2.9},
   \begin{equation}\label{con57}
\sum_{1\le i\le k\le p-1}\frac{H_{k,2}}{ik}=\sum_{k=1}^{p-1}\frac{H_{k,2}}{k}
\sum_{i=1}^{k}\frac{1}{i}=\sum_{k=1}^{p-1}\frac{H_{k,2}\cdot H_k}{k}
\equiv 0\pmod{p}.
    \end{equation}
Now inserting  (\ref{con54})--(\ref{con57}) into 
(\ref{con53}), we obtain
\begin{equation}\label{con58}
  \sum_{1\le i\le k\le p-1}\frac{(-1)^{k-1}}{ik}{p-1\choose k}
\equiv -\frac{1}{2}H_{p-1,2}+p\sum_{k=1}^{p-1}\frac{H_k^2}{k}.
   \end{equation}
The equality (\ref{con52}) and the congruence (\ref{con58}) give 
 \begin{equation}\label{con59}
 \sum_{k=1}^{p-1}\frac{H_k^2}{k}
\equiv\frac{3}{2p}\sum_{k=1}^{p-1}\frac{1}{k^2}\pmod{p^2}.
   \end{equation}
The congruences (\ref{con59}) and (\ref{con50}) of Lemma~\ref{l2.10}
yield 
\begin{equation}\label{con60}
 \sum_{k=1}^{p-1}\frac{H_k^2}{k}
\equiv -\frac{3}{p^2}\sum_{k=1}^{p-1}\frac{1}{k}\pmod{p^2}.
   \end{equation}
Finally, the  congruences (\ref{con59}) and (\ref{con60})
complete  proof of Theorem~\ref{t1.1}. 
 \end{proof}

\noindent{\bf Remark.}  From the identity
     $$
H_{k-1}^3=\left(H_k-\frac{1}{k}\right)^3=H_k^3-3\frac{H_k^2}{k}+
3\frac{H_k}{k^2}-\frac{1}{k^3}
   $$
immediately follows that 
 $$
\sum_{k=1}^{p-1}\frac{H_k^2}{k}-\sum_{k=1}^{p-1}\frac{H_k}{k^2}
=\frac{1}{3}\left(H_{p-1}^3-H_{p-1,3}\right).
   $$
Inserting in the right hand  side of the above identity 
the congruences $H_{p-1}\equiv 0\,(\bmod{\,p^2})$ and 
$H_{p-1,3}\equiv -\frac{6p^2B_{p-5}}{5}\,(\bmod{\,p^3})$
from \cite[Theorem 5.1(a) with  $k=3$]{s1}, we find that
for a prime $p\ge 7$,  
  $$
\sum_{k=1}^{p-1}\frac{H_k^2}{k}-\sum_{k=1}^{p-1}\frac{H_k}{k^2}
\equiv \frac{2p^2B_{p-5}}{5}\pmod{p^3}.
  $$
However,  the determination of 
$\sum_{k=1}^{p-1}\frac{H_k^2}{k}\,(\bmod{\,p^3})$ 
it seems  to be a difficult problem.

\end{document}